\documentclass[11pt, a4paper]{article}
\usepackage{mathrsfs, mathtools, amsthm, amssymb, thm-restate}  
\usepackage{paralist, tablists}  

\usepackage{graphicx, xcolor, tikz}  
\usetikzlibrary{calc, shapes, decorations.pathreplacing, decorations.markings}  
\usepackage{caption}  
\usepackage[labelformat=simple]{subcaption}  

\usepackage[left=25mm, top=25mm, bottom=25mm, right=25mm]{geometry}  
\usepackage[T1]{fontenc} 

\usepackage[inline]{enumitem}  
\usepackage[square, numbers, sort&compress]{natbib}  
\usepackage[
  bookmarks=true,                
  bookmarksnumbered=true,        
  bookmarksopen=true,            
  plainpages=false,              
  colorlinks=true,               
  linkcolor=blue,                
  citecolor=black,               
  anchorcolor=green,             
  urlcolor=blue,                 
  hyperindex=true                
]{hyperref} 

\newtheorem{theorem}{Theorem}[section] 
\newtheorem{lemma}[theorem]{Lemma}
\newtheorem{corollary}[theorem]{Corollary}
\newtheorem{case}{Case}
\newtheorem{subcase}{Subcase}[case]

\newtheorem{claim}{Claim}
\theoremstyle{definition}
\newtheorem{definition}{Definition}

\newtheorem{proposition}{Proposition}

\makeatletter
\def\th@plain{%
  \upshape 
}
\makeatother

\newcommand{\etal}{et~al.\ }
\newcommand{\ie}{i.e.,\ }

\usepackage{marvosym,wasysym}
\usepackage{array}
\usepackage{bm}  

\makeatletter
\renewenvironment{proof}[1][\proofname]{\par
  \pushQED{\qed}%
  \normalfont \topsep6\p@\@plus6\p@\relax
  \trivlist
  \item[\hskip\labelsep
        \bfseries
    #1\@addpunct{.}]\ignorespaces
}{%
  \popQED\endtrivlist\@endpefalse
}
\makeatother

\usepackage[capitalise]{cleveref}  
\crefname{claim}{Claim}{Claims}

\begin{document}

\title{Planar graphs without 4-, 7-, 9-cycles and 5-cycles normally adjacent to 3-cycles}
\author{Zhengjiao Liu\footnote{School of Mathematics and Statistics, Henan University, Kaifeng, 475004, P. R. China} \and Tao Wang\footnote{Center for Applied Mathematics, Henan University, Kaifeng, 475004, P. R. China. {\tt wangtao@henu.edu.cn; https://orcid.org/0000-0001-9732-1617} } \and Xiaojing Yang\footnote{School of Mathematics and Statistics, Henan University, Kaifeng, 475004, P. R. China. {\tt Corresponding
author: yangxiaojing@henu.edu.cn}}}
\date{July 13, 2024}
\maketitle
	
\begin{abstract}
A graph is \emph{$(\mathcal{I}, \mathcal{F})$-partitionable} if its vertex set can be partitioned into two parts such that one part $\mathcal{I}$ is an independent set, and the other $\mathcal{F}$ induces a forest. A graph is \emph{$k$-degenerate} if every subgraph $H$ contains a vertex of degree at most $k$ in $H$. Bernshteyn and Lee defined a generalization of $k$-degenerate graphs, which is called \emph{weakly $k$-degenerate}. In this paper, we show that planar graphs without $4$-, $7$-, $9$-cycles, and $5$-cycles normally adjacent to $3$-cycles are both $(\mathcal{I}, \mathcal{F})$-partitionable and weakly $2$-degenerate.

\end{abstract}

\section{Introduction}
Graphs considered in this paper are finite and simple. A graph is \emph{$k$-degenerate} if every subgraph $H$ contains a vertex of degree at most $k$ in $H$. A graph $G$ is \emph{$(a, b)$-partitionable} if its vertices can be partitioned into two subsets, with one subset inducing an $a$-degenerate subgraph and the other inducing a $b$-degenerate subgraph of $G$. It is worth noting that planar graphs are $5$-degenerate. Thomassen \cite{MR1358992, MR1866722} established that such graphs are $(1, 2)$-partitionable and $(0,3)$-partitionable. 

A graph is \emph{$(\mathcal{I}, \mathcal{F})$-partitionable} if its vertices can be divided into two parts, one part $\mathcal{I}$ forming an independent set, and the other part $\mathcal{F}$ inducing a forest. Note that an independent set is characterized by being $0$-degenerate, while a forest is $1$-degenerate. It is evident that a graph is $(\mathcal{I}, \mathcal{F})$-partitionable if and only if it is $(0, 1)$-partitionable. Borodin and Glebov \cite{MR1918259} established that every planar graph of girth at least $5$ is $(\mathcal{I}, \mathcal{F})$-partitionable. An \emph{$l$-cycle} is a cycle of length $l$. Liu and Yu \cite{MR4115511} confirmed that planar graphs without $4$-, $6$- and $8$-cycles are $(\mathcal{I}, \mathcal{F})$-partitionable. Recently, Kang \etal \cite{arXiv:2303.04648} proved that every planar graph without $4$-, $6$- and $9$-cycles is $(\mathcal{I}, \mathcal{F})$-partitionable. Two cycles are \emph{normally adjacent} if their intersection is isomorphic to the complete graph $K_{2}$. In a plane graph, two faces are \emph{adjacent} if their boundaries share at least one edge, while two faces are \emph{normally adjacent} if their bounded cycles are normally adjacent. In this paper, we prove that planar graphs without $4$-, $7$-, $9$-cycles, and $5$-cycles normally adjacent to $3$-cycles are also $(\mathcal{I}, \mathcal{F})$-partitionable.

\begin{theorem}\label{Liu:1}
Planar graphs without $4$-, $7$-, $9$-cycles, and $5$-cycles normally adjacent to $3$-cycles are $(\mathcal{I}, \mathcal{F})$-partitionable.
\end{theorem}

\begin{corollary}\label{Liu:2}
Every planar graph without $4$-, $6$-, $7$-, and $9$-cycles is $(\mathcal{I}, \mathcal{F})$-partitionable.
\end{corollary}

To establish \cref{Liu:1}, we rely on the following local structural result. 
\begin{theorem}\label{Local}
If $G$ is a plane graph without $4$-, $7$-, $9$-cycles, and $5$-cycles normally adjacent to $3$-cycles, then one of the following statements holds:
\begin{enumerate}[label = (\alph*)]
\item $\delta(G) \leq 2$;
\item there exists a subgraph isomorphic to a configuration in \cref{PP}. 
\end{enumerate}
\end{theorem}

Indeed, \cref{Local} serves not only as a crucial tool in establishing \cref{Liu:1}, but also finds application in proving a result on weakly $2$-degenerate.

\begin{figure}
\centering
\subcaptionbox{A 10-face normally adjacent to a 3-face.\label{10Cap3Fig}}[0.45\linewidth]
{\begin{tikzpicture}
\def\s{0.707}
\foreach \ang in {1, 2, 3, 4, 5, 6, 7, 8, 9, 10}
{
\def\pointname{v\ang}
\coordinate (\pointname) at ($(\ang*360/10+54:\s)$);
}
\draw (v1)--(v2)--(v3)--(v4)--(v5)--(v6)--(v7)--(v8)--(v9)--(v10)--cycle;
\coordinate (S) at ($(v4)!1!60:(v3)$);
\draw (v3)--(S)--(v4);
\node[circle, inner sep = 1, fill, draw] () at (S) {};
\foreach \ang in {1, 2, 3, 4, 5, 6, 7, 8, 9, 10}
{
\node[circle, inner sep = 1, fill, draw] () at (v\ang) {};
}
\end{tikzpicture}}
\subcaptionbox{Two adjacent $10$-faces.\label{BAD10FACE}}[0.45\linewidth]
{\begin{tikzpicture}
\def\s{0.707}
\foreach \ang in {1, 2, 3, 4, 5, 6, 7, 8, 9, 10}
{
\def\pointname{v\ang}
\coordinate (\pointname) at ($(\ang*360/10+54:\s)$);
}
\draw (v1)--(v2)--(v3)--(v4)--(v5)--(v6)--(v7)--(v8)--(v9)--(v10)--cycle;
\foreach \ang in {1, 2, 3, 4, 5, 6, 7, 8, 9, 10}
{
\node[circle, inner sep = 1, fill, draw] () at (v\ang) {};
}

\foreach \ang in {1, 2, 3, 4, 5, 6, 7, 8, 9, 10}
{
\def\pointname{u\ang}
\coordinate (\pointname) at ($(\ang*360/10+54:\s) + (1.9*\s, 0)$);
}
\draw (u1)--(u2)--(u3)--(u4)--(u5)--(u6)--(u7)--(u8)--(u9)--(u10)--cycle;
\foreach \ang in {1, 3, 4, 6, 7, 8, 9, 10}
{
\node[circle, inner sep = 1, fill, draw] () at (u\ang) {};
}
\draw (v10)--(u2);
\draw (v7)--(u5);
\node[rectangle, inner sep = 2, fill, draw] () at (u2) {};
\node[rectangle, inner sep = 2, fill, draw] () at (u5) {};
\end{tikzpicture}}\vspace{0.5cm}

\subcaptionbox{An 8-face normally adjacent to a 5-face.\label{8Cap5Fig}}[0.45\linewidth]
{\begin{tikzpicture}
\def\s{0.707}
\foreach \ang in {1, 2, ..., 8}
{
\def\pointname{v\ang}
\coordinate (\pointname) at ($(\ang*360/8+22.5:\s)$);
}
\draw (v1)--(v2)--(v3)--(v4)--(v5)--(v6)--(v7)--(v8)--cycle;
\coordinate (w5) at ($(v4)!1!90:(v3)$);
\coordinate (w1) at ($(w5)!1!150:(v4)$);
\coordinate (w2) at ($(v3)!1!-90:(v4)$);
\draw (v4)--(w5)--(w1)--(w2)--(v3);
\node[circle, inner sep = 1, fill, draw] () at (w5) {};
\node[circle, inner sep = 1, fill, draw] () at (w1) {};
\node[circle, inner sep = 1, fill, draw] () at (w2) {};
\foreach \ang in {1, 2, ..., 8}
{
\node[circle, inner sep = 1, fill, draw] () at (v\ang) {};
}
\end{tikzpicture}}
\caption{Configurations, where a solid point represents a $3$-vertex, and a rectangle represents a $4$-vertex.}
\label{PP}
\end{figure}

Next, we introduce another graph property: weakly $f$-degenerate. 
\begin{definition}[\textsf{Delete} operation]
Let $G$ be a graph and $f : V(G) \longrightarrow \mathbb{Z}$ be a function. For a vertex $u \in V(G)$, the operation $\mathsf{Delete}_{[u]}(G, f)$ yields the graph $G' = G - u$ and the function $f': V(G') \longrightarrow \mathbb{Z}$ defined as 
    \begin{align*}
		f'(v) \coloneqq
		&\begin{cases}
			f(v) - 1, & \text{if $uv \in E(G)$};\\[0.3cm]
			f(v), & \text{otherwise}.
		\end{cases}
	\end{align*}
An application of the \textsf{Delete} operation is \emph{legal} if the resulting function $f'$ is nonnegative. 
\end{definition}

\begin{definition}[\textsf{DeleteSave} operation]
Let $G$ be a graph and $f : V(G) \longrightarrow \mathbb{Z}$ be a function. Given a pair of adjacent vertices $u, w \in V(G)$, the operation  $\mathsf{DeleteSave}_{[u; w]}(G, f)$ produces the graph $G' = G - u$ and the function $f' : V(G') \longrightarrow \mathbb{Z}$ defined as
	\begin{align*}
		f'(v) \coloneqq
		&\begin{cases}
			f(v) - 1, & \text{if $v \in N(u) - \{w\}$};\\[0.3cm]
			f(v), & \text{otherwise}.
		\end{cases}
	\end{align*}
An application of the \textsf{DeleteSave} operation is \emph{legal} if $f(u) > f(w)$ and the resulting function $f'$ is nonnegative. 
\end{definition}

A graph $G$ is \emph{weakly $f$-degenerate} if it is possible to remove all vertices from $G$ by a sequence of legal applications of the Delete or DeleteSave operations. A graph is \emph{$f$-degenerate} if it is weakly $f$-degenerate with only \textsf{Delete} operations involved. Clearly, for any $d \in \mathbb{N}$, a graph is \emph{$d$-degenerate} if and only if it is $f$-degenerate with respect to the constant function of value $d$. The \emph{degeneracy}  $\mathsf{d}(G)$ of $G$ is the least $d$ such that $G$ is $d$-degenerate, and the \emph{weak degeneracy}  $\mathsf{wd}(G)$ of $G$ is the minimum integer $d$ such that $G$ is weakly $d$-degenerate. 

Bernshteyn and Lee \cite{MR4606413} established the following relations among some graph coloring parameters.
\begin{proposition}\label{prop}
For any graph $G$, we always have 
\[
\chi(G) \leq \chi_{\ell}(G) \leq \chi_{\textsf{DP}}(G) \leq \chi_{\textsf{DPP}}(G) \leq \textsf{wd}(G) + 1 \leq \textsf{d}(G) + 1, 
\]
where $\chi_{\textsf{DP}}(G)$ is the DP-chromatic number of $G$, and $\chi_{\textsf{DPP}}(G)$ is the DP-paint number of $G$. 
\end{proposition}

It is interesting to investigate whether known upper bounds on $\chi_{\ell}(G), \chi_{\textsf{DP}}(G)$ and $\chi_{\textsf{DPP}}(G)$ are true for  $\mathsf{wd}(G) + 1$. Han \etal \cite{MR4663366} proved that planar graphs without $3$-, $6$- and $7$-cycles are weakly $2$-degenerate. An \emph{$l^{-}$-cycle} is a cycle of length at most $l$. 

\begin{theorem}[Han \etal \cite{MR4663366}]
If $G$ is a triangle-free plane graph in which no $4$-cycle is normally adjacent to a $5^{-}$-cycle, then $G$ is weakly $2$-degenerate.
\end{theorem}

Wang \cite{MR4564473} considered the weak degeneracy of planar graphs without $4$- and $6$-cycles.

\begin{theorem}[Wang \cite{MR4564473}]\label{w4v}
\text{}
\begin{enumerate}[label = (\arabic*)]
\item Let $G$ be a planar graph without $4$-, $6$- and $9$-cycles. If there are no $7$-cycles normally adjacent to $5$-cycles, then $G$ is weakly $2$-degenerate.

\item Let $G$ be a planar graph without $4$-, $6$- and $8$-cycles. If there are no $3$-cycles normally adjacent to $9$-cycles, then $G$ is weakly $2$-degenerate.
\end{enumerate}
\end{theorem}

Some classes of weakly $3$-degenerate graphs are discussed in \cite{Wang2019+,MR4746933}. Various further generalizations of weakly $f$-degenerate can be found in \cite{ZZZ2023}.

In this paper, we establish the following result.

\begin{theorem}\label{WD-MAIN}
If $G$ is a plane graph without $4$-, $7$-, $9$-cycles, and $5$-cycles normally adjacent to $3$-cycles, then $G$ is weakly $2$-degenerate.
\end{theorem}

\begin{corollary}
If $G$ is a plane graph without $4$-, $6$-, $7$-, $9$-cycles, then $G$ is weakly $2$-degenerate.
\end{corollary}

We provide a proof for \cref{Local} in \cref{sec:2}, and present a detailed proof for \cref{Liu:1} in \cref{sec:3}. The proof of \cref{WD-MAIN} can be found in \cref{sec:4}.

\section{Proof of \cref{Local}}
\label{sec:2}
For simplicity, let us focus on connected graphs. Otherwise, we consider each component individually. Suppose, for the sake of contradiction, that $G$ is a counterexample to the statement. Therefore, $G$ is a plane graph without $4$-, $7$-, $9$-cycles, and $5$-cycles normally adjacent to $3$-cycles. However, the minimum degree is at least $3$, and there exists no subgraph isomorphic to any configuration as depicted in \cref{PP}. 

Considering the graph's structure, we have the following results. While some of them are straightforward, we omit their proofs. In the remainder of this section, we use $b(f)$ to denote the boundary walk of a face $f$. A \emph{$k$-face} (\emph{$k^{+}$-face}, or \emph{$k^{-}$-face}) is a face of size exactly $k$ (at least $k$, or at most $k$). Similarly, a \emph{$k$-vertex} (\emph{$k^{+}$-vertex}, or \emph{$k^{-}$-vertex}) is a vertex of degree exactly $k$ (at least $k$, or at most $k$). A vertex is \emph{incident with} a face if it lies on the boundary of the face.

\begin{lemma}\label{lem:s}
The following statements are true.
  \begin{enumerate}[label = (\roman*)]
	\item\label{no-chord5} A $5$-cycle has no chords. 
	
	\item\label{no-chord6} A $6$-cycle has no chords. 
	
	\item The boundary walk of a $6$-face consists of a $6$-cycle or two $3$-cycles.
	
	\item\label{5va6} A $5$-face cannot be adjacent to a $6$-face.
	
	\item\label{5va5} Two adjacent $5$-faces must be normally adjacent. 
	
	\item\label{2FiveFaces} A $3$-vertex is incident with at most two $5$-faces.
	
	\item\label{no7face} There are no $7$-faces.
	
	\item The boundary of an $8$-face consists of an $8$-cycle, or two $3$-cycles and a cut edge, or one $3$-cycle and one $5$-cycle.
	
	\item The boundary of a $9$-face consists of one $3$-cycle and one $6$-cycle, or three $3$-cycles.
	
	\item\label{3va10} Every face adjacent to a $3$-face is a $10^{+}$-face.
 \end{enumerate}
\end{lemma}

\begin{proof}[Proof of \ref{5va6}]
Let $f = [u_{1}u_{2}u_{3}v_{1}v_{2}]$ be a $5$-face, and $g$ be an adjacent $6$-face. If $b(g)$ consists of two $3$-cycles, then $b(f)$ would be adjacent to a $3$-cycle, which is a contradiction. So we assume that $b(g)$ is a $6$-cycle $v_{1}v_{2}v_{3}v_{4}v_{5}v_{6}v_{1}$. If $b(f)$ and $b(g)$ are normally adjacent, then $b(f)$ and $b(g)$ form a $9$-cycle, resulting in a contradiction. This implies that $\{u_{1}, u_{2}, u_{3}\} \cap \{v_{3}, v_{4}, v_{5}, v_{6}\} \neq \emptyset$. Since $\delta(G) \geq 3$, we conclude that $u_{1} \neq v_{3}$ and $u_{3} \neq v_{6}$. It follows from \ref{no-chord5} and \ref{no-chord6} that $\{v_{3}, v_{6}\} \cap \{u_{1}, u_{2}, u_{3}\} = \emptyset$ and $\{u_{1}, u_{3}\} \cap \{v_{3}, v_{4}, v_{5}, v_{6}\} = \emptyset$. The only remaining consideration is whether $u_{2}$ is identified with $v_{4}$ or $v_{5}$. If $u_{2} = v_{4}$, then $v_{2}u_{1}v_{4}v_{3}v_{2}$ is a $4$-cycle, a contradiction. By symmetry, $u_{2}$ cannot be identified with $v_{5}$.
\end{proof}

\begin{proof}[Proof of \ref{5va5}]
Consider two adjacent $5$-faces, $f = [u_{1}u_2u_3v_{1}v_{2}]$ and $g = [v_{1}v_{2}v_{3}v_{4}v_{5}]$. Since $\delta(G) \geq 3$ and $G$ has no $4$-cycles, we confirm that $\{u_{1}, u_{3}\} \cap \{v_{3}, v_{4}, v_{5}\} = \emptyset$ and $\{v_{3}, v_{5}\} \cap \{u_{1}, u_{2}, u_{3}\} = \emptyset$. If $u_{2} = v_{4}$, then $v_{2}u_{1}v_{4}v_{3}v_{2}$ is a $4$-cycle, leading to a contradiction. 
\end{proof}

\begin{proof}[Proof of \ref{2FiveFaces}]
Suppose that there exists a $3$-vertex incident with three $5$-faces $f, g$ and $h$. It follows from \ref{5va5} that the three $5$-faces are pairwise normally adjacent. Consequently, the boundaries of these three faces form a $9$-cycle, resulting in a contradiction.
\end{proof}

\begin{proof}[Proof of \ref{3va10}]
We observe that if two $3$-faces are adjacent, then $G$ would contain a $4$-cycle. Since there are no $4$-cycles, it follows that a $3$-face is not adjacent to a $4$-face. Additionally, a $3$-face is not adjacent to a $5$-face. 

Consider a $5$-face $f=[v_1v_2v_3v_4v_5]$. Suppose a $3$-face $h=[v_1v_2u_1]$ is adjacent to $f$. Since the minimum degree is at least $3$, we have $u_1 \neq v_3, v_5$. Furthermore, $u_{1} \neq v_{4}$; otherwise, $v_2v_1v_5v_4v_2$ is a $4$-cycle. Thus, a $3$-face cannot be adjacent to a $5$-face. 

Consider a $6$-face $f=[v_1v_2v_3v_4v_5v_6]$. Suppose a $3$-face $h=[v_1v_6u_1]$ is adjacent to $f$. If $f$ is normally adjacent to $h$, then a $7$-cycle is formed. Since the minimum degree is at least $3$, we have $u_1 \neq v_{2}, v_{5}$. If $u_1 = v_3$ or $u_1 = v_4$, then $v_6v_3v_4v_5v_6$ or $v_1v_2v_3v_4v_1$ is a $4$-cycle. There are no $3$-faces adjacent to $6$-faces when the boundary of the $6$-face consists of two $3$-cycles. Thus, a $3$-face is not adjacent to a $6$-face. 

Note that the boundary of every $7$-face consists of a $7$-cycle, so there is no $3$-face adjacent to a $7$-face. 

Consider an $8$-face $f=[v_1v_2v_3v_4v_5v_6v_7v_8]$, where the boundary of $f$ is an $8$-cycle. Suppose that a $3$-face $h=[v_1v_8u_1]$ is adjacent to $f$. A $9$-cycle is formed if $f$ is normally adjacent to $h$. Since the minimum degree is at least $3$, we have $u_1 \neq v_2, v_7$. If $u_1 = v_3$ or $u_1 = v_4$, then $v_3v_2v_1v_8v_3$ or $v_1v_2v_3v_4v_1$ forms a $4$-cycle. By symmetry, $u_1 \neq v_6$ and $u_1 \neq v_5$. Assuming a $3$-face $h$ is adjacent to an $8$-face $f$ with the boundary consisting of a $3$-cycle and a $5$-cycle, or two $3$-cycles and a cut edge, it cannot be included in another $3$-cycle that is part of the boundary of $f$. Thus, there are no $3$-faces adjacent to $8$-faces. 

Note that the boundary of a $9$-face $f$ consists of three $3$-cycles, or one $3$-cycle and one $6$-cycle, then no edge on the boundary of $f$ can be included in other $3$-cycles since $G$ contains no 4-, 7-cycles and no 5-cycles normally adjacent to 3-cycles. Thus, there are no $3$-faces adjacent to $9$-faces. Therefore, every $3$-face is adjacent to three $10^{+}$-faces.
\end{proof}

For a face $f$, we use $d(f)$ to denote the degree (or size) of $f$. A \emph{bad face} $f$ is a $10$-face $[v_{1}v_{2}\dots v_{10}]$ incident with ten distinct $3$-vertices and two $3$-faces $[v_{1}v_{2}u_{1}]$ and $[v_{3}v_{4}u_{3}]$. Notably, a bad face is bounded by a $10$-cycle. Since no configuration as depicted in \cref{10Cap3Fig} exists, each of $u_{1}$ and $u_{3}$ is a $4^{+}$-vertex. The edge $v_{2}v_{3}$ is referred to as a \emph{good edge}, $u_{1}v_{2}v_{3}u_{3}$ is a \emph{good path}, and the face $g$ incident with $u_{1}v_{2}v_{3}u_{3}$ is a \emph{good face}. According to \cref{lem:s}\ref{3va10}, every good face is a $10^{+}$-face.

Assign an initial charge $\mu(x)$ to each element $x \in V(G) \cup F(G)$ as $\mu(x) = 2d(x) - 6$ for each $x \in V(G)$ and $\mu(x) = d(x) - 6$ for each $x \in F(G)$, where $F(G)$ denotes the face set of $G$. Utilizing Euler's formula and the Handshaking Lemma, the sum of the initial charges is $-12$. We will move the charges around so that each vertex and face ends with a non-negative final charge, thereby reaching a contradiction. Let $\mu'(f)$ denote the final charge after the discharging procedure.

A $3$-vertex is \emph{bad} if it is incident with a $3$-face, \emph{worse} if it is incident with precisely one $5$-face, and \emph{worst} if it is incident with precisely two $5$-faces. A $5$-face is \emph{light} if it is incident with five $3$-vertices.

We employ the following discharging rules:

\begin{enumerate}[label = \textbf{R\arabic*.}, ref = R\arabic*]
\item\label{R1} Each $3$-face receives $1$ from each incident vertex.
\item\label{R2} Let $v$ be a $3$-vertex.
\begin{enumerate}
	\item\label{R2a} If $v$ is a bad vertex, then it receives $\frac{1}{2}$ from each incident $10^{+}$-face.
	\item\label{R2b} If $v$ is a worse vertex incident with a light $5$-face, then it sends $\frac{1}{3}$ to the incident light $5$-face and receives $\frac{1}{6}$ from each incident $8^{+}$-face.
	\item\label{R2c} If $v$ is a worse vertex incident with a non-light $5$-face, then it sends $\frac{1}{8}$ to the incident non-light $5$-face and receives $\frac{1}{16}$ from each incident $8^{+}$-face.
	\item\label{R2d} If $v$ is a worst vertex and the two $5$-faces are light, then it sends $\frac{1}{6}$ to each incident light $5$-face and receives $\frac{1}{3}$ from the incident $8^{+}$-face.
	\item\label{R2e} If $v$ is a worst vertex and the two $5$-faces are non-light, then it sends $\frac{1}{8}$ to each incident non-light $5$-face and receives $\frac{1}{4}$ from the incident $8^{+}$-face.
	\item\label{R2f} If $v$ is a worst vertex adjacent to precisely one light $5$-face, then it sends $\frac{1}{6}$ to the light $5$-face, $\frac{1}{8}$ to the non-light $5$-face, and receives $\frac{1}{3}$ from the incident $8^{+}$-face.
\end{enumerate}	

\item\label{R3} Let $v$ be a $4$-vertex.
\begin{enumerate}
	\item\label{R3a} If $v$ is incident with precisely one $3$-face and precisely one $5$-face, then it sends $1$ to the $5$-face.
	\item\label{R3b} If $v$ is incident with precisely one $3$-face and no $5$-face, then it sends $\frac{1}{2}$ to each incident face adjacent to the $3$-face.
	\item\label{R3c} If $v$ is incident with four $5^{+}$-faces, then it sends $\frac{1}{2}$ to each incident face.
\end{enumerate}

\item\label{R4} Each $5^{+}$-vertex sends $\frac{2}{3}$ to each incident $5^{+}$-face.
\item\label{R5} Let $\mu^{*}$ denote the charges after applying rules \ref{R1}--\ref{R4}. Each good face $g$ sends $\frac{\mu^{*}(g)}{t}$ to each adjacent bad face, where $t$ is the number of adjacent bad faces.
\end{enumerate}

\begin{lemma}\label{lem:rule}
Assume $f = [v_{1}v_{2} \dots v_{10}]$ is a bad $10$-face with associated $3$-faces $[v_{1}v_{2}u_{1}]$ and $[v_{3}v_{4}u_{3}]$. Let $g$ be the face incident with the path $u_{1}v_{2}v_{3}u_{3}$. If both $u_{1}$ and $u_{3}$ are $4$-vertices, then $g$ sends at least $\frac{1}{4}$ to $f$ by \ref{R5}. If $u_{1}$ is a $4$-vertex and $u_{3}$ is a $5^{+}$-vertex, then $g$ sends at least $\frac{1}{2}$ to $f$ by \ref{R5}.
\end{lemma}

\begin{proof}
By \cref{lem:s}, $g$ is a $10^{+}$-face. We consider the following three cases according to the size of $g$.
\begin{case}
$g$ is a $10$-face. 
\end{case}
Then $g$ is incident with at most three good edges. 
\begin{subcase}
Both $u_{1}$ and $u_{3}$ are $4$-vertices.
\end{subcase}
Since there is no configuration as depicted in \cref{BAD10FACE}, $g$ is incident with at most seven $3$-vertices. If $g$ is incident with precisely one good edge, then $\mu^{*}(g) \geq 10 - 6 - 7 \times \frac{1}{2} = \frac{1}{2}$ by \ref{R2}, and $g$ sends at least $\frac{1}{2}$ to $f$ by \ref{R5}. If $g$ is incident with precisely two good edges, then it is incident with at most seven $3$-vertices, and $\mu^{*}(g) \geq 10 - 6 - 7 \times \frac{1}{2} = \frac{1}{2}$ by \ref{R2}, and $g$ sends at least $\frac{1}{4}$ to $f$ by \ref{R5}. If $g$ is incident with precisely three good edges, then $g$ is incident with at least four $4^{+}$-vertices, then $\mu^{*}(g) \geq 10 - 6 - 6 \times \frac{1}{2} = 1$ by \ref{R2}, and $g$ sends at least $\frac{1}{3}$ to $f$ by \ref{R5}.

\begin{subcase}
$u_{1}$ is a $4$-vertex and $u_{3}$ is a $5^{+}$-vertex.
\end{subcase}
If $g$ is incident with precisely one good edge, then it is incident with at most eight $3$-vertices, and $\mu^{*}(g) \geq 10 - 6 - 8 \times \frac{1}{2} + \frac{2}{3} = \frac{2}{3}$ by \ref{R2} and \ref{R4}, and $g$ sends at least $\frac{2}{3}$ to $f$ by \ref{R5}. If $g$ is incident with precisely two good edges, then it is incident with at most seven $3$-vertices, and $\mu^{*}(g) \geq 10 - 6 - 7 \times \frac{1}{2} + \frac{2}{3} = \frac{7}{6}$, and $g$ sends at least $\frac{7}{12}$ to $f$ by \ref{R5}. If $g$ is incident with exactly three good edges, then it is incident with at most six $3$-vertices, and $\mu^{*}(g) \geq 10 - 6 - 6 \times \frac{1}{2} + \frac{2}{3} = \frac{5}{3}$, and $g$ sends at least $\frac{5}{9}$ to $f$ by \ref{R5}.

\begin{case}
$g$ is an $11$-face. 
\end{case}
Then $g$ is incident with at most three good edges. 
\begin{subcase}
Both $u_{1}$ and $u_{3}$ are $4$-vertices.
\end{subcase}
If $g$ is incident with precisely one good edge, then $\mu^{*}(g) \geq 11 - 6 - 9 \times \frac{1}{2} = \frac{1}{2}$ by \ref{R2}, and $g$ sends at least $\frac{1}{2}$ to $f$ by \ref{R5}. If $g$ is incident with precisely two good edges, then $\mu^{*}(g) \geq 11 - 6 - 8 \times \frac{1}{2} = 1$ by \ref{R2}, and $g$ sends at least $\frac{1}{2}$ to $f$ by \ref{R5}. If $g$ is incident with precisely three good edges, then $\mu^{*}(g) \geq 11 - 6 - 7 \times \frac{1}{2} = \frac{3}{2}$, and $g$ sends at least $\frac{1}{2}$ to $f$ by \ref{R5}. 

\begin{subcase}
$u_{1}$ is a $4$-vertex and $u_{3}$ is a $5^{+}$-vertex.
\end{subcase}
If $g$ is incident with precisely one good edge, then $\mu^{*}(g) \geq 11 - 6 - 9 \times \frac{1}{2} + \frac{2}{3} = \frac{7}{6}$ by \ref{R2} and \ref{R4}. Thus, $g$ sends at least $\frac{7}{6}$ to $f$ by \ref{R5}. If $g$ is incident with precisely two good edges, then it is associated with at least three vertices of degree $4$ or higher, implying $\mu^{*}(g) \geq 11 - 6 - 8 \times \frac{1}{2} + \frac{2}{3} = \frac{5}{3}$, and $g$ sends at least $\frac{5}{6}$ to $f$ by \ref{R5}. If $g$ is incident with exactly three good edges, then $\mu^{*}(g) \geq 11 - 6 - 7 \times \frac{1}{2} + \frac{2}{3} = \frac{13}{6}$, and $g$ sends at least $\frac{13}{18}$ to $f$ by \ref{R5}. 

\begin{case}
$g$ is a $12^{+}$-face. 
\end{case}

If $g$ is incident with precisely $t$ good edges, then $g$ is incident with at most $d(g) - t$ vertices of degree $3$. Thus, $\mu^{*}(g) \geq d(g) - 6 - (d(g) - t) \times \frac{1}{2} = \frac{d(g)-12}{2} + \frac{t}{2} \geq \frac{t}{2}$. Therefore, $g$ sends at least $\frac{1}{2}$ to $f$ by \ref{R5}. This completes the proof of \cref{lem:rule}.
\end{proof}

\begin{claim}\label{Claim:1}
	Each vertex in $G$ has a non-negative final charge.
\end{claim}

\begin{proof}[Proof of \cref{Claim:1}]
Let $v$ be a $3$-vertex. If $v$ is incident with a $3$-face, then it is incident with two $10^{+}$-faces by \cref{lem:s}\ref{3va10}, thus $v$ is a bad vertex and $\mu'(v) = -1 + \frac{1}{2} \times 2 = 0$ by \ref{R1} and \ref{R2a}. If $v$ is a worse vertex incident with a light $5$-face, then the other two incident faces are $8^{+}$-faces by \cref{lem:s}, thus $\mu'(v) = - \frac{1}{3} + \frac{1}{6} \times 2 = 0$ by \ref{R2b}. If $v$ is a worse vertex incident with a non-light $5$-face, then $\mu'(v) = - \frac{1}{8} + \frac{1}{16} \times 2 = 0$ by \ref{R2c}. If $v$ is a worst vertex incident with two light $5$-faces, then the remaining incident face is an $8^{+}$-face by \cref{lem:s}, implying $\mu'(v) = -\frac{1}{6} \times 2 + \frac{1}{3} \times 1 = 0$ by \ref{R2d}. If $v$ is a worst vertex incident with two non-light $5$-faces, then $\mu'(v) = -\frac{1}{8} \times 2 + \frac{1}{4} \times 1 = 0$ by \ref{R2e}. If $v$ is a worst vertex incident with precisely one light $5$-face, then $v$ sends $\frac{1}{6}$ to the light $5$-face and $\frac{1}{8}$ to the non-light $5$-face by \ref{R2f}, thus $\mu'(v) = -\frac{1}{6} -\frac{1}{8} + \frac{1}{3} > 0$. Note that $v$ cannot be incident with three $5$-faces. For all other cases, $v$ does not send out any charge, so $\mu'(v) = \mu(v) = 0$.

Let $v$ be a $4$-vertex. Since there are no adjacent $3$-faces, $v$ is incident with at most two $3$-faces. Furthermore, a $3$-face is adjacent to three $10^{+}$-faces by \cref{lem:s}\ref{3va10}. If $v$ is incident with a $3$-face and another $5^{-}$-face, then $\mu'(v) = 2 - 1 \times 2 = 0$ by \ref{R1} and \ref{R3a}. If $v$ is incident with precisely one $3$-face and no $5$-face, then $\mu'(v) = 2 - 1 - \frac{1}{2} \times 2 = 0$ by \ref{R1} and \ref{R3b}. If $v$ is incident with four $5^{+}$-faces, then $\mu'(v) = 2 - \frac{1}{2} \times 4 = 0$ by \ref{R3c}.

Let $v$ be a $5^{+}$-vertex. Since there are no adjacent $3$-faces, $v$ is incident with at most $\lfloor \frac{d(v)}{2} \rfloor$ triangular-faces. Thus, $\mu'(v) \geq 2 d(v) - 6 - 1 \times \lfloor \frac{d(v)}{2} \rfloor - \frac{2}{3} \times (d(v) - \lfloor \frac{d(v)}{2} \rfloor) \geq 0$. Hence, every vertex ends with a non-negative charge.
\end{proof}

\begin{claim}\label{Claim:2}
	Each face in $G$ has a non-negative final charge.
\end{claim}

\begin{proof}[Proof of \cref{Claim:2}]
Let $f$ be an arbitrary face in $G$. Note that there are no $4$- or $7$-faces.
\setcounter{case}{0}
\begin{case}
$f$ is a $3$-face.
\end{case}
Then $\mu'(f) = - 3 + 1 \times 3 = 0$ by \ref{R1}. 

\begin{case}
$f$ is a $5$-face.
\end{case}
First, we assume that $f$ is a light $5$-face, \ie every vertex on $f$ is a $3$-vertex in $G$. By \cref{lem:s}\ref{2FiveFaces}, $f$ is adjacent to at most two $5$-faces. Then $f$ contains at least one worse vertex. By \ref{R2b}, \ref{R2d}, and \ref{R2f}, $f$ receives $\frac{1}{3}$ from each incident worse vertex and $\frac{1}{6}$ from each incident worst vertex. Hence, $\mu'(f) \geq 5 - 6 + \frac{1}{3} \times 1 + \frac{1}{6} \times 4 = 0$. 
Second, assume that $f$ is a non-light $5$-face incident with a $4^{+}$-vertex. Then $f$ receives at least $\frac{1}{2}$ from each incident $4^{+}$-vertex by \ref{R3} and \ref{R4}, and $\frac{1}{8}$ from each incident $3$-vertex by \ref{R2}. Thus, $\mu'(f) \geq 5 - 6 + \frac{1}{2} \times 1 + \frac{1}{8} \times 4 = 0$. 

\begin{case}
$f$ is a $6$-face.
\end{case}
By \cref{lem:s}\ref{3va10} and \ref{5va6}, each face adjacent to $f$ must be a $6^{+}$-face. Thus, $f$ does not send out any charge by applying the rules. It follows that $\mu'(f) \geq 0$. 

\begin{case}
$f$ is an $8$-face.
\end{case}
By \cref{lem:s}\ref{3va10}, $f$ is not adjacent to any $3$-face, and no vertex on $f$ is bad. By \ref{R2}, $f$ sends at most $\frac{1}{3}$ to each incident $3$-vertex. By \ref{R3}, $f$ does not send any charge to incident $4$-vertices, but it probably receives some charge from incident $4$-vertices. 
\begin{subcase}
$b(f)$ is an $8$-cycle.
\end{subcase}
First, assume that every vertex on $f$ is a $3$-vertex. Since there is no configuration as depicted in \cref{8Cap5Fig}, $f$ cannot be adjacent to any light $5$-face. By \ref{R2}, $f$ sends at most $\frac{1}{4}$ to each incident $3$-vertex. Thus, $\mu'(f) \geq 8 - 6 - \frac{1}{4} \times 8 = 0$. Next, assume that $f$ contains at least one $4^{+}$-vertex. If $f$ could receive at least $\frac{1}{3}$ in total from incident $4^{+}$-vertices, then $\mu'(f)\geq 8 - 6 - \frac{1}{3} \times 7 + \frac{1}{3} = 0$. We may assume that $f$ receives less than $\frac{1}{3}$ in total from incident $4^{+}$-vertices. By \ref{R4}, $f$ does not contain $5^{+}$-vertices. Let $v$ be a $4$-vertex on $f$, and $u, w$ be the two neighbors along the boundary of $f$. By \ref{R3c}, $v$ is not incident with four $5^{+}$-faces; otherwise, $f$ receives $\frac{1}{2}$ from $v$, contradicting the fact that $f$ receives less than $\frac{1}{3}$ in total from all incident $4^{+}$-vertices. Thus, $v$ is incident with a $3$-face. By \cref{lem:s}\ref{3va10}, each of $uv$ and $vw$ is incident with a $10^{+}$-face. Hence, neither $u$ nor $w$ is a worst vertex. By \ref{R2}, $f$ sends at most $\frac{1}{6}$ to each of $u$ and $w$. Therefore, $\mu'(f) \geq 8 - 6 - \frac{1}{3} \times 5 - \frac{1}{6} \times 2 = 0$. 

\begin{subcase}
$b(f)$ consists of two $3$-cycles and a cut edge, or one $3$-cycle and one $5$-cycle. 
\end{subcase}

Let $xyzx$ be a $3$-cycle, where $z$ is a cut-vertex. Since $\delta(G) \geq 3$ and $G$ has no $4$-cycles, we have that none of $xy, yz$ and $xz$ is incident with a $3$-face. As every $5$-cycle has no chords, and no $3$-cycle is normally adjacent to a $5$-cycle, we conclude that none of $xy, yz$ and $xz$ is incident with a $5$-face. Hence, none of $xy, yz$ and $xz$ is incident with a $5^{-}$-face. According to the discharging rules, $f$ does not send out charge to $x$ or $y$. Therefore, $\mu'(f) \geq 8 - 6 - \frac{1}{3} \times 6 = 0$. 

\begin{case}
$f$ is a $9$-face.
\end{case}
By \cref{lem:s}\ref{3va10}, $f$ is not adjacent to any $3$-face, and no vertex on $f$ is bad. According to the discharging rules, $f$ sends at most $\frac{1}{3}$ to each incident vertex. Hence, $\mu'(f) \geq 3 - \frac{1}{3} \times 9 = 0$. 

\begin{case}
$f$ is a $10$-face.
\end{case}

\begin{subcase}
$f$ contains at least two $4^{+}$-vertices. 
\end{subcase}
According to the discharging rules, $f$ sends at most $\frac{1}{2}$ to each incident $3$-vertex but no charge to $4^{+}$-vertices. Thus, $\mu^{*}(f) \geq 10 - 6 - \frac{1}{2} \times 8 = 0$, implying $\mu'(f) \geq 0$ by \ref{R5}. 

\begin{subcase}
$f$ contains no $4^{+}$-vertices. 
\end{subcase}

By parity, both the number of bad vertices and worse vertices on $f$ are even. Firstly, suppose that $f$ contains ten bad vertices. Then $b(f)$ must be a $10$-cycle. Since there exists no configuration as depicted in \cref{10Cap3Fig}, $f$ contains five good edges. Thus, $f$ sends at most $\frac{1}{2}$ to each incident $3$-vertex by \ref{R2} and receives at least $\frac{1}{4}$ from each adjacent good face by \cref{lem:rule}. Hence, $\mu'(f) \geq 10 - 6 - \frac{1}{2} \times 10 + \frac{1}{4} \times 5 > 0$. 

Secondly, assume that $f$ contains a worse vertex and precisely eight bad vertices. Thus, $b(f)$ must be a $10$-cycle. By parity, $f$ contains two adjacent worse vertices, and three good edges. Thus, $\mu'(f) \geq 10 - 6 - \frac{1}{2} \times 8 - \frac{1}{6} \times 2 + \frac{1}{4} \times 3 > 0$. 

Thirdly, assume that $f$ contains precisely eight bad vertices and no worse vertices. Observe that $f$ does not contain any worst vertex. Then $\mu'(f) \geq 10 - 6 - \frac{1}{2} \times 8 = 0$. 

Finally, assume that $f$ contains precisely $k$ bad vertices, where $k \leq 6$. If $k \leq 4$, then $\mu'(f) \geq 10 - 6 - \frac{1}{2} \times k - \frac{1}{3} \times (10 - k) = \frac{4 - k}{6} \geq 0$. For $k = 6$, by parity, $f$ contains at most two worst vertices. Thus, $\mu'(f) \geq 10 - 6 - \frac{1}{2} \times 6 - \frac{1}{6} \times 2 - \frac{1}{3} \times 2 = 0$. 

\begin{subcase}
$f$ contains precisely one $4^{+}$-vertex. 
\end{subcase}

Then $f$ cannot be a good face in this case. If $f$ contains a $5^{+}$-vertex, then $\mu'(f) \geq 10 - 6 - \frac{1}{2} \times 9 + \frac{2}{3} > 0$. Assume that $f$ contains nine $3$-vertices and one $4$-vertex $w$. If $f$ can receive at least $\frac{1}{2}$ from $w$, then $\mu'(f) \geq 10 - 6 - \frac{1}{2} \times 9 + \frac{1}{2} = 0$. However, assuming $f$ receives less than $\frac{1}{2}$ from $w$, we find that $f$ receives zero from $w$ by \ref{R3}. If $f$ contains at most six bad vertices, then $\mu'(f) \geq 10 - 6 - \frac{1}{2} \times 6 - \frac{1}{3} \times 3 = 0$. 

Now, suppose that $f$ contains precisely seven bad vertices. Since $7$ is odd, $w$ is on an adjacent $3$-face. Since $f$ receives zero from $w$, the $4$-vertex $w$ is incident with a $5$-face. By \ref{R3a}, $f$ sends nothing to $w$. Note that there exists a $3$-vertex $u$ that is worse or incident with three $6^{+}$-faces, then $f$ sends at most $\frac{1}{6}$ to $u$. Thus, $\mu'(f) \geq 10 - 6 - \frac{1}{2} \times 7 - \frac{1}{6} \times 1 - \frac{1}{3} \times 1 = 0$.

Next, suppose that $f$ contains precisely eight bad vertices. Let $x$ be the remaining non-bad $3$-vertex. Observe that $x$ cannot be a worst vertex. Assume $x$ is a worse vertex incident with a $5$-face $g$. Then $x$ and $w$ are consecutive vertices on $g$. It follows that the eight bad vertices must be on four $3$-faces. Hence, the $4$-vertex $w$ must be incident with four $5^{+}$-faces, and $f$ receives $\frac{1}{2}$ from $w$ by \ref{R3c}, which is a contradiction. If $x$ is not incident with any $5^{-}$-face, then $f$ does not send any charge to $x$. Hence, $\mu'(f) = 10 - 6 - \frac{1}{2} \times 8 = 0$. 

Finally, suppose that $f$ is incident with nine bad vertices. Then $f$ is incident with five $3$-faces, and the $4$-vertex $w$ is incident with one $3$-face and three $10^{+}$-faces. By \ref{R3}, $w$ sends $\frac{1}{2}$ to $f$, contradicting the fact that $f$ receives zero from $w$. 

\begin{case}
$f$ is a $11$-face.
\end{case}

Since $11$ is odd, $f$ is incident with at most ten bad vertices. If $f$ is incident with ten bad vertices, then the remaining vertex is a $4^{+}$-vertex or a $3$-vertex incident with three $10^{+}$-faces, so $\mu^{*}(f) = 11 - 6 - \frac{1}{2} \times 10 = 0$. If $f$ is incident with at most eight bad vertices, then $\mu^{*}(f) \geq 11 - 6 - \frac{1}{2} \times 8 - \frac{1}{3} \times 3 = 0$. In the remaining case, we consider that $f$ is incident with precisely nine bad vertices. Since $9$ is odd, there exists a $4^{+}$-vertex on an adjacent $3$-face. Thus, $\mu^{*}(f) \geq 11 - 6 - \frac{1}{2} \times 10 = 0$. In each subcase, $\mu^{*}(f) \geq 0$, thus $\mu'(f) \geq 0$ by \ref{R5}. 

\begin{case}
$f$ is a $12^{+}$-face.
\end{case}
Recall that $f$ sends at most $\frac{1}{2}$ to each incident vertex. Then $\mu^{*}(f) \geq d(f) - 6 - \frac{1}{2} \times d(f) \geq 0$. Hence, $\mu'(f) \geq 0$ by \ref{R5}.
\end{proof}

\section{Proof of \cref{Liu:1}}
\label{sec:3}
To prove our main result \cref{Liu:1}, we establish a stronger theorem by introducing additional necessary definitions and notations. 

A \emph{cover} of a graph $G$ is a graph $H$ with vertex set $V(H) = \bigcup_{v \in V(G)} X_{v}$, where $X_{v} = \{(v, 1), (v, 2), \dots, (v, s)\}$, and edge set $\mathscr{M} = \bigcup_{uv \in E(G)}\mathscr{M}_{uv}$, where $\mathscr{M}_{uv}$ is a matching between $X_{u}$ and $X_{v}$. Notably, $X_{v}$ is an independent set in $H$, and $\mathscr{M}_{uv}$ may be an empty set. A vertex subset $T \subseteq V(H)$ is a \emph{transversal} of $H$ if $|T \cap X_{v}| = 1$ for each $v \in V(G)$. 

Let $H$ be a cover of $G$ and $f$ be a function mapping $V(H)$ to $\{0, 1, 2, \dots\}$, we refer to the pair $(H, f)$ as a \emph{valued cover} of $G$. For any vertex $x = (v, i) \in V(H)$, we simply write $f(x)$ or $f(v, i)$ interchangeably for $f((v, i))$. Let $S$ be a subset of $V(G)$, we use $H_{S}$ to denote the induced subgraph $H[\bigcup_{v \in S} X_{v}]$. A transversal $T$ is a \emph{strictly $f$-degenerate transversal} if every subgraph $K$ of $H[T]$ has a vertex $x \in K$ with $\deg_{K}(x) < f(x)$.

\begin{theorem}[Lu \etal \cite{MR4357325}]\label{mindeg}
Let $G$ be a graph, and $(H, f)$ be a valued cover of $G$. If $(H, f)$ has no strictly $f$-degenerate transversal, but $(H - H_{x}, f)$ has one, then $\deg_{G}(x) \geq f(v, 1) + \dots + f(v, s)$. 
\end{theorem}

Define 
\[
\mathscr{D} \coloneqq \{\,v \mid f(v, 1) + f(v, 2) + \dots + f(v, s) \geq \deg_{G}(v)\,\}.
\]

\begin{theorem}[Lu \etal \cite{MR4357325}]\label{Gallai-SFDT}
Let $G$ be a graph, and $(H, f)$ be a valued cover of $G$. Assume $B$ is a nonempty subset of $\mathscr{D}$ with $G[B]$ having no cut vertex. If $(H, f)$ has no strictly $f$-degenerate transversal, but $(H - H_{B}, f)$ has one, then $G[B]$ is a cycle or a complete graph, or $\deg_{G[B]}(v) \leq \max_{q} \big\{f(v, q)\big\}$ for each $v \in B$. 
\end{theorem}

\begin{theorem}[Wang \etal \cite{Wang2019+}]\label{WW}
Let $k$ be an integer with $k \geq 3$, and $K$ be an induced subgraph of $G$ with vertices ordered as $v_{1}, v_{2}, \dots, v_{m}$, satisfying the following conditions:
\begin{enumerate}[label = (\roman*)]
\item\label{WW-1} $k - (d_{G}(v_{1}) - d_{K}(v_{1})) > k - (d_{G}(v_{m}) - d_{K}(v_{m}))$. 
\item\label{WW-2} $d_{G}(v_{m}) \leq k$ and $v_{1}v_{m} \in E(G)$. 
\item\label{WW-3} For $2 \leq i \leq m - 1$, $v_{i}$ has at most $k - 1$ neighbors in $G - \{v_{i+1}, \dots, v_{m}\}$.
\end{enumerate}
Let $H$ be a cover of $G$, and $f$ be a function mapping $V(H)$ to $\{0, 1, 2\}$. If $f(v, 1) + \dots + f(v, s) \geq k$ for each vertex $v \in V(G)$, then any strictly $f$-degenerate transversal of $H - \bigcup_{v \in V(K)}L_{v}$ can be extended to that of $H$. \qed
\end{theorem}

Let $G$ be a graph, and let $H$ be a cover of $G$. We say that $H$ is a \emph{canonical cover} of $G$ if the cover has the property that $(u, i)(v, j) \in E(H)$ if and only if $uv \in E(G)$ and $i = j$. Observe that a canonical cover of $G$ is isomorphic to $s$ copies of $G$.

\begin{theorem}\label{SFDT}
Let $G$ be a plane graph without $4$-, $7$-, $9$-cycles and $5$-cycles normally adjacent to $3$-cycles. Assume $H$ is a canonical cover of $G$ with $s = 2$, and $f$ is a special function with $f(v, 1) = 1$ and $f(v, 2) = 2$ for every vertex $v \in V(G)$. Then $H$ has a strictly $f$-degenerate transversal. 
\end{theorem}
\begin{proof}
Let $G$ be a counterexample with the minimum number of vertices. In other words, $(H, f)$ has no strictly $f$-degenerate transversal, but $(H_{S}, f)$ has one for every proper subset $S \subset V(G)$. According to \cref{Local}, $G$ has a vertex of degree at most $2$, or there is a subgraph isomorphic to one of the configurations in \cref{PP}. By \cref{mindeg}, the minimum degree of $G$ is at least $3$. 

Suppose there exists a subgraph isomorphic to a configuration in \cref{10Cap3Fig} or \cref{8Cap5Fig}. Let $U$ be the vertex set of the subgraph. Note that $G[U]$ is neither a cycle nor a complete graph. Moreover, there is a vertex $v$ in $G[U]$ with degree greater than $2$. This contradicts \cref{Gallai-SFDT}. 

Suppose there is a subgraph $K$ isomorphic to the configuration in \cref{BAD10FACE}. We use the labels as in \cref{CF}. Order the vertices of $K$ as
\[
x_{3}, y_{3}, y_{4}, y_{5}, y_{6}, y_{7}, y_{8}, y_{9}, y_{10}, x_{2}, x_{1}, x_{10}, x_{9}, x_{8}, x_{7}, x_{6}, x_{5}, x_{4}.
\]
We can check that the list above satisfies the conditions in \cref{WW}.
By minimality, $(H - H_K, f)$ has a strictly $f$-degenerate transversal $T$. According to \cref{WW}, we can obtain a strictly $f$-degenerate transversal of $H$ by extending $T$, which leads to a contradiction.
\end{proof}

\begin{proof}[Proof of \cref{Liu:1}]
Let $s = 2$, and $H$ be a canonical cover of $G$. Let $f$ be a mapping with $f(v, 1) = 1$ and $f(v, 2) = 2$ for every vertex $v \in V(G)$. By \cref{SFDT}, $H$ has a strictly $f$-degenerate transversal $T$. Let $\mathcal{I} = \{v \mid (v, 1) \in T\}$ and $\mathcal{F} = \{v \mid (v, 2) \in T\}$. Observe that $\mathcal{I}$ is an independent set in $G$, and $\mathcal{F}$ induces a forest in $G$. Then $G$ is $(\mathcal{I}, \mathcal{F})$-partitionable, this completes the proof of \cref{Liu:1}. 
\end{proof}

\section{Proof of \cref{WD-MAIN}}
\label{sec:4}
We say that $G$ is a \emph{minimal} graph of weak degeneracy $d$ if $\mathsf{wd}(G) = d$ and $\mathsf{wd}(H) < d$ for every proper subgraph $H$ of $G$. A connected graph is a \emph{GDP-tree} if every block is either a cycle or a complete graph. We need the following Gallai-type result established by Bernshteyn and Lee~\cite{MR4606413}.

\begin{theorem}[Bernshteyn and Lee~\cite{MR4606413}]\label{Gallai}
Let $G$ be a minimal graph with weak degeneracy $d \geq 3$. Then the following statements hold.
\begin{enumerate}[label = (\roman*)]
\item\label{G1} The minimum degree of $G$ is at least $d$. 
\item\label{G2} Let $U \subseteq \{u \in V(G) \mid d_{G}(u) = d\}$. Then every component of $G[U]$ is a GDP-tree.
\end{enumerate}
\end{theorem}

Now, we can easily prove \cref{WD-MAIN}. 
\begin{proof}[Proof of \cref{WD-MAIN}]
Let $G$ be a counterexample to \cref{WD-MAIN} such that every proper subgraph $H$ of $G$ has $\mathsf{wd}(H) \leq 2$. Observe that $G$ is a minimal graph with weak degeneracy $3$. By \cref{Gallai}\ref{G1}, the minimum degree of $G$ is at least three. By \cref{Gallai}\ref{G2}, there are no subgraphs isomorphic to the configurations depicted in \cref{10Cap3Fig} or \cref{8Cap5Fig}. 

\begin{figure}
\centering
\begin{tikzpicture}[line width = 1pt]
\def\s{1}
\foreach \ang in {1, 2, 3, 4, 5, 6, 7, 8, 9, 10}
{
\def\pointname{v\ang}
\coordinate (\pointname) at ($(\ang*360/10+54:\s)$);
}
\draw (v1)node[above]{\small$x_{4}$}--(v2)node[above]{\small$x_{5}$}--(v3)node[left]{\small$x_{6}$}--(v4)node[left]{\small$x_{7}$}--(v5)node[below]{\small$x_{8}$}--(v6)node[below]{\small$x_{9}$}--(v7)node[below]{\small$x_{10}$}--(v8)node[left]{\small$x_{1}$}--(v9)node[left]{\small$x_{2}$}--(v10)node[above]{\small$x_{3}$}--cycle;
\foreach \ang in {1, 2, 3, 4, 5, 6, 7, 8, 9, 10}
{
\node[circle, inner sep = 1, fill, draw] () at (v\ang) {};
}

\foreach \ang in {1, 2, 3, 4, 5, 6, 7, 8, 9, 10}
{
\def\pointname{u\ang}
\coordinate (\pointname) at ($(\ang*360/10+54:\s) + (1.9*\s, 0)$);
}
\draw (u1)node[above]{\small$y_{4}$}--(u2)node[above]{\small$y_{3}$}--(u3)--(u4)--(u5)node[below]{\small$y_{10}$}--(u6)node[below]{\small$y_{9}$}--(u7)node[below]{\small$y_{8}$}--(u8)node[right]{\small$y_{7}$}--(u9)node[right]{\small$y_{6}$}--(u10)node[above]{\small$y_{5}$}--cycle;
\foreach \ang in {1, 3, 4, 6, 7, 8, 9, 10}
{
\node[circle, inner sep = 1, fill, draw] () at (u\ang) {};
}
\draw (v10)--(u2);
\draw (v7)--(u5);
\node[rectangle, inner sep = 2, fill, draw] () at (u2) {};
\node[rectangle, inner sep = 2, fill, draw] () at (u5) {};
\end{tikzpicture}
\caption{A configuration with labels.}
\label{CF}
\end{figure}
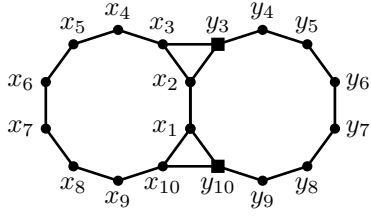

Let $W$ be the set of vertices represented in \cref{CF}. Since $G$ contains no $4$-, $7$-, $9$-cycles, and no $5$-cycles normally adjacent to $3$-cycles, $x_{4}$ has only two neighbors $x_{3}$ and $x_{5}$ in $G[W]$. By minimality, we can first remove all vertices from $G-W$ through a sequence of legal applications of \textsf{Delete} or \textsf{DeleteSave} operations. Let $f'$ be the resulting function on $W$. Note that $f'(x_{4}) = 2$ and $f'(x_{3}) = 3$. Next, we apply a legal application of $\mathsf{DeleteSave}_{[x_{3};x_{4}]}(G[W], f')$ operation, and then we remove all the remaining vertices with the following order and \textsf{Delete} operations: 
\[
y_{3}, y_{4}, y_{5}, y_{6}, y_{7}, y_{8}, y_{9}, y_{10}, x_{2}, x_{1}, x_{10}, x_{9}, x_{8}, x_{7}, x_{6}, x_{5}, x_{4}.
\]
Hence, $G$ is weakly $2$-degenerate, leading to a contradiction. Therefore, there are no subgraphs isomorphic to the configuration in \cref{BAD10FACE} in $G$. However, this contradicts \cref{Local}. 
\end{proof}

\vskip 0mm \vspace{0.3cm} \noindent{\bf Acknowledgments.} We thank the two anonymous referees for their valuable comments and constructive suggestions on the manuscript. The second author was supported by the Natural Science Foundation of Henan Province (No. 242300420238). The third author was supported by National Natural Science Foundation of China (No. 12101187).

\end{document}